\documentclass [12pt]{amsart}
\usepackage{amsmath,amssymb}
 \usepackage{amsthm, amsfonts}
 \usepackage{enumerate}
 \usepackage{amssymb,amsmath,amsthm, amsfonts}
 \usepackage{cite}
 \usepackage{xcolor}
\usepackage[pagewise]{lineno}
\newtheorem{theorem}{Theorem}[section]

\newtheorem{corollary}[theorem]{Corollary}
\theoremstyle{definition}
\newtheorem{definition}[theorem]{Definition}
\newtheorem{example}[theorem]{Example}

\theoremstyle{remark}

\numberwithin{equation}{section}

\begin{document}

\title [{{On $gr$-$C$-$2^{A}$-secondary submodules }}]{On $Gr$-$C$-$2^{A}$-secondary submodules }

 \author[{{T. Alwardat, K. Al-Zoubi and M. Al-Dolat}}]{\textit{Thikrayat Alwardat, Khaldoun Al-Zoubi}* and Mohammed Al-Dolat  }

\address
{\textit{ Thikrayat Alwardat, Department of Mathematics and
Statistics, Jordan University of Science and Technology, P.O.Box
3030, Irbid 22110, Jordan.}}
\bigskip
{\email{\textit{tdalwardat21@sci.just.edu.jo}}}

\address
{\textit{ Khaldoun Al-Zoubi , Department of Mathematics and
Statistics, Jordan University of Science and Technology, P.O.Box
3030, Irbid 22110, Jordan.}}
\bigskip
{\email{\textit{kfzoubi@just.edu.jo}}}

\address
{\textit{Mohammed Al-Dolat , Department of Mathematics and
Statistics, Jordan University of Science and Technology, P.O.Box
3030, Irbid 22110, Jordan.}}
\bigskip
{\email{\textit{mmaldolat@just.edu.jo}}}


\date{}
\begin{abstract}
Let $\Omega$ be a group with identity $e$, $\Gamma$ be a $\Omega$-graded commutative ring and $\Im$ a graded $\Gamma$-module. In this article, we introduce the concept of $gr$-$C$-$2^{A}$-secondary submodules and investigate some properties of this new class of graded submodules. A non-zero graded submodule $S$ of $\Im$ is said to be a $gr$-$C$-$2^{A}$-secondary submodule if whenever $r,s \in h(\Gamma)$, $L$ is a graded submodule of $\Im$, and $rs\,S\subseteq L$, then either $r\,S\subseteq L$ or $s\,S \subseteq L$ or $rs \in Gr(Ann_\Gamma(S))$.

 \end{abstract}
\keywords{ graded classical 2-absorbing secondary submodules, graded 2-absorbing submodules, graded 2-absorbing primary submodules.\\
$*$ Corresponding author}
 \maketitle


\section{Introduction and Preliminaries}
Throughout this article, we assume that $\Gamma$ is a commutative $\Omega$-graded ring with identity and $\Im$ is a unitary graded $\Gamma$-Module.

Let $\Omega$ be a group with identity $e$ and $\Gamma$ be a commutative ring with
identity $1_{\Gamma}$. Then $\Gamma$ is a \emph{$\Omega$-graded ring} if there exist
additive subgroups $\Gamma_{g}$ of $\Gamma$ such that $\Gamma=\bigoplus_{g\in \Omega}\Gamma_{g}$ and $%
\Gamma_{g}\Gamma_{h}\subseteq \Gamma_{gh}$ for all $g,h\in \Omega$. Moreover, $h(\Gamma)=\bigcup_{g\in \Omega}\Gamma_{g}$, (see \cite{NVM}).

A Left $\Gamma$-module $\Im$ is said to be $\Omega$\textit{-graded }$\Gamma$\textit{-module}  if there exists a family of
additive subgroups $\{\Im_{\alpha }\}_{\alpha \in \Omega}$ of $\Im$ such that $%
\Im=\bigoplus_{\alpha \in \Omega}\Im_{\alpha }$ and $\Gamma_{\alpha }\Im_{\beta }\subseteq
\Im_{\alpha \beta }$ for all $\alpha ,\beta \in \Omega$. Also if an element of $\Im$ belongs to $\cup
_{\alpha \in \Omega}\Im_{\alpha }=h(\Im)$, then it is called a homogeneous. We refer to \cite{HG, NVG, NVG2,
NVM} for basic properties and more information on graded rings and graded
modules.  By $L\leq _{\Omega}\Im$, we mean that $L$ is a $\Omega$-graded
submodule of $\Im.$

Let $\Gamma$ be a $\Omega$-graded ring, $\Im$ a graded $\Gamma$-module and $S$ a graded
submodule of $\Im$. Then $(S:_{\Gamma}\Im)$ is defined as $(S:_{\Gamma}\Im)=\{a\in
\Gamma|a\Im\subseteq S\}.$ The annihilator of $\Im$ is defined as $(0:_{\Gamma}\Im)$
and is denoted by $Ann_{\Gamma}(\Im).$ Let $\Gamma$ be a $\Omega$-graded ring. The \textit{%
graded radical} of a graded ideal $L$, denoted by $Gr(L)$, is the set of all
$t=\sum\nolimits_{\alpha\in \Omega}t_{\alpha}\in \Gamma$ such that for each $\alpha\in \Omega$ there
exists $n_{\alpha}>0$ with $t_{\alpha}^{n_{\alpha}}\in L$, (see \cite{RAO}).
A proper graded submodule $S$ of $\Im$ is said to be \textit{a
completely graded irreducible} if $S=\cap _{\alpha \in \Delta }S_{\alpha },$
where $\{S_{\alpha }\}_{\alpha \in \Delta }$ is a family of graded
submodules of $\Im$, then $S=S_{\beta }$ for some $\beta \in \Delta .$ The notion of graded 2-absorbing ideals was introduced and studied in \cite{KACO}. Al-Zoubi and Abu-Dawwas in \cite{KAO} extended graded 2-absorbing ideals to graded
2-absorbing submodules. In \cite{KN}, the authors introduced the concept of graded 2-absorbing primary ideal which is a generalization of
graded primary ideal. The notion of graded 2-absorbing primary submodules as a generalization of
graded 2-absorbing primary ideals was introduced and studied in \cite{CO}. In \cite{KM, RR}, the authors introduced the dual notion of graded 2-absorbing submodules
(that is, graded 2-absorbing (resp., graded strongly 2-absorbing) second
submodules) of $\Im$, and investigated some properties of these classes of
graded modules. In this paper, we introduce the concept of graded classical 2-absorbing
secondary submodules as a dual notion of graded 2-absorbing
primary submodules. We investigate the basic properties and
characteristics of graded classical 2-absorbing secondary submodules.
 \section{Results}

\begin{definition}
 Let $\Gamma$ be a $\Omega$-graded ring and $\Im$ a graded $\Gamma$-module. A non-zero graded submodule $S$ of $\Im$ is said to be graded classical 2-absorbing secondary (Abbreviated, $gr$-$C$-$2^{A}$-secondary) submodule of $\Im$ if whenever $r,s \in h(\Gamma)$, $L\leq _{\Omega}\Im$, and $rs\,S\subseteq L$, then $r\,S\subseteq L$ or $s\,S \subseteq L$ or $rs \in Gr(Ann_\Gamma(S))$.

 We say that $\Im$ is a $gr$-$C$-$2^{A}$-secondary module if $\Im$ is a  $gr$-$C$-$2^{A}$-secondary submodule of itself.
\end{definition}


\begin{theorem}\label{th2.3}
Let $S$ be a $gr$-$C$-$2^{A}$-secondary submodule of $\Im$, let $I = \bigoplus_{\alpha \in \Omega} I_\alpha$  and $J = \bigoplus_{\alpha \in \Omega} J_\alpha$ be a graded ideals of $\Gamma$. Then for every $\alpha,\beta \in \Omega$ and $L\leq _{\Omega}\Im$, with $I_\alpha J_\beta S \subseteq L$ either $I_\alpha S \subseteq L$ or $J_\beta S \subseteq L$ or $I_\alpha J_\beta \subseteq Gr(Ann_\Gamma(S))$.
\end{theorem}

\begin{proof}
Let $\alpha,\beta \in \Omega$ such that $I_\alpha J_\beta S \subseteq L$ for some $L\leq _{\Omega}\Im$. Assume that $I_\alpha J_\beta \nsubseteq Gr(Ann_\Gamma(S)$. Then  there exist $r_\alpha \in I_\alpha$ and $s_\beta \in J_\beta$ such that $r_\alpha s_\beta \notin Gr(Ann_\Gamma(S)$. Now since $r_\alpha s_\beta \, S \subseteq L$, we get  $r_\alpha \, S \subseteq L$ or  $s_\beta \, S \subseteq L$. We show that either  $I_\alpha S \subseteq L$ or $J_\beta S \subseteq L$. On contrary, we suppose that  $I_\alpha S \nsubseteq L$ and $J_\beta S \nsubseteq L$. Then there exist $r' _\alpha \in I_\alpha$ and $s' _\beta \in J_\beta$ such that $r' _\alpha \, S \nsubseteq L$ and $s' _\beta \, S \nsubseteq L$. Since $r' _\alpha s' _\beta \, S \subseteq L$ and  $S$ be a $gr$-$C$-$2^{A}$-secondary submodule of $\Im$, $r' _\alpha s' _\beta \in Gr(Ann_\Gamma(S)$. We have three cases:

 Case I: Suppose that $r_\alpha \, S \subseteq L$ but $s_\beta \, S \nsubseteq L$. Since $r'_\alpha s_\beta \, S \subseteq L$  and $s_\beta \, S \nsubseteq L$ and $r'_\alpha \, S \nsubseteq L$, this implies $r'_\alpha s_\beta \in Gr(Ann_\Gamma(S))$. Since  $r_\alpha \, S \subseteq L$ and $r'_\alpha \, S \nsubseteq L$, we get $(r_\alpha + r'_\alpha) \, S \nsubseteq L$. As $(r_\alpha + r'_\alpha)  s_\beta \, S \subseteq L$ and $s_\beta \, S \nsubseteq L$, then $(r_\alpha + r'_\alpha) \, S \nsubseteq L$ implies $(r_\alpha + r'_\alpha)  s_\beta \in Gr(Ann_\Gamma(S)$. Since $r'_\alpha s_\beta \in Gr(Ann_\Gamma(S))$, we get $r_\alpha s_\beta \in Gr(Ann_\Gamma(S))$, a contradiction.

 Case II: Suppose $s_\beta \, S \subseteq L$ but  $r_\alpha \, S \nsubseteq L$. Then similar to the Case I, we get a contradiction.

 Case III: Suppose $r_\alpha \, S \subseteq L$ and $s_\beta \, S \subseteq L$. Now  $s_\beta \, S \subseteq L$ and  $s'_\beta \, S \nsubseteq L$ imply $(s_\beta + s'_\beta) \, S \nsubseteq L$. Since  $r'_\alpha(s_\beta + s'_\beta) \, S \subseteq L$ and  $(s_\beta + s'_\beta) \, S \nsubseteq L$ and $r'_\alpha \, S \nsubseteq L$, we get $r'_\alpha(s_\beta + s'_\beta) \in Gr(Ann_\Gamma(S))$. Now as  $r'_\alpha  s'_\beta \in Gr(Ann_\Gamma(S))$, we get  $r'_\alpha s_\beta  \in Gr(Ann_\Gamma(S))$. Again $r_\alpha \, S \subseteq L$ and  $r'_\alpha \, S \nsubseteq L$ imply $(r_\alpha + r'_\alpha)\, S \nsubseteq L$. Since  $(r_\alpha + r'_\alpha)s'_\beta\, S \subseteq L$ and  $(r_\alpha + r'_\alpha)\, S \nsubseteq L$ and $s'_\beta \, S \nsubseteq L$, we have   $(r_\alpha + r'_\alpha)s'_\beta \in Gr(Ann_\Gamma(S))$. Since $ r'_\alpha s'_\beta \in Gr(Ann_\Gamma(S))$,  we get  $ r_\alpha s'_\beta \in Gr(Ann_\Gamma(S))$. Since  $(r_\alpha + r'_\alpha)(s_\beta + s'_\beta) \, S \subseteq L$ and $(r_\alpha + r'_\alpha)\, S \nsubseteq L$ and  $(s_\beta + s'_\beta) \, S \nsubseteq L$, we get   $(r_\alpha + r'_\alpha)(s_\beta + s'_\beta) \in Gr(Ann_\Gamma(S))$. Since $r_\alpha s'_\beta, r'_\alpha s_\beta, r'_\alpha s'_\beta \in Gr(Ann_\Gamma(S))$, we have  $ r_\alpha s_\beta \in Gr(Ann_\Gamma(S))$, a contradiction. Thus  $I_\alpha S \subseteq L$ or $J_\beta S \subseteq L$.
\end{proof}


 \begin{theorem}\label{thm2.4}
Let $S$ be a $gr$-$C$-$2^{A}$-secondary submodule of $\Im$, then for each $a,b \in h(\Gamma)$ we have $abS = a \, S$ or $ab \, S = bS$ or $ab \in Gr(Ann_\Gamma(S))$.
 \end{theorem}

\begin{proof}
 Let $a,b \in h(\Gamma)$, then $abS \subseteq abS$ implies that $aS \subseteq abS$ or $aS \subseteq abS$ or $ab \in Gr(Ann_\Gamma(S))$. Clearly  $abS \subseteq aS$ and  $abS \subseteq bS$, so we have $abS = aS$ or $abS = bS$ or $ab \in Gr(Ann_\Gamma(S))$.
\end{proof}

Let $U$ and $P$ are two graded submodule of an graded $\Gamma$-module.
To prove $U\subseteq P$, its enough to show that if $V$ is a completely graded irreducible submodule of $\Im$ such that  $P \subseteq V$, then  $U\subseteq V$, see\cite{KM}. A proper graded ideal $L$ of $\Gamma$ is called a graded 2-absorbing primary (Abbreviated, $gr$-$2^{A}$-primary) ideal if whenever $a,b,c\in h(\Gamma)$ with $abc\in L$, then $%
ab\in L$ or $ac\in Gr(L)$ or $bc\in Gr(L)$.
 \begin{theorem}\label{thm 2.5}
 Let $S$ be a  $gr$-$C$-$2^{A}$-secondary submodule of a graded $\Gamma$-module $\Im$. Then $Ann_\Gamma(S)$ is a $gr$-$2^{A}$-primary ideal of $\Gamma$.
 \end{theorem}

 \begin{proof}
  Let $r, s, t \in h(\Gamma)$ wit $rst \in Ann_\Gamma(S)$. Assume that $rs \notin Ann_\Gamma(S)$ and $rt \notin Gr(Ann_\Gamma(S))$. We show that $st \in  Gr(Ann_\Gamma(S))$. There exist completely irreducible submodule $J_1$ and $J_2$ of $\Im$ such that $rs \, S \nsubseteq J_1$ and $rt \, S \nsubseteq J_2$. Since $rst \, S =0 \subseteq J_1 \cap J_2$, $st \, S \subseteq (J_1 \cap J_2 :_\Im r)$. Since  $S$ is $gr$-$C$-$2^{A}$-secondary submodule of $\Im$, we have $rs \, S \subseteq J_1 \cap J_2$ or $rt \, S \subseteq J_1 \cap J_2$ or $st \in  Gr(Ann_\Gamma(S))$. If $rs \, S \subseteq J_1 \cap J_2$ or $rt \, S \subseteq J_1 \cap J_2$, then $rs \, S \subseteq J_1$ or $rt \, S \subseteq J_2$ which are contradictions. Therefore $st \in  Gr(Ann_\Gamma(S))$.
 \end{proof}
A proper graded ideal $L$ of $\Gamma$ is said to be a graded 2-absorbing (Abbreviated, $gr$-$2^{A}$)
ideal of $\Gamma$ if whenever $a,b,c\in h(\Gamma)$ with $abc\in L$, then $ab\in L$ or
$ac\in L$ or $bc\in L$, see \cite{KACO}.

 \begin{corollary}\label{2.6}
 Let $S$ be a  $gr$-$C$-$2^{A}$-secondary submodule of a graded $\Gamma$-module $\Im$. Then $Gr(Ann_\Gamma(S))$ is a  $gr$-$2^{A}$ ideal of $\Gamma$.
 \end{corollary}

 \begin{proof}
 By Theorem \ref{thm 2.5}, $Ann_\Gamma(S)$ is $gr$-$2^{A}$-primary ideal of $\Gamma$. So by [\cite{KN}, Theorem 2.3], $Gr(Ann_\Gamma(S))$ is  $gr$-$2^{A}$ ideal of $\Gamma$.
 \end{proof}
 \vspace{0.2cm}
 The following example shows that the converse of Theorem \ref{thm 2.5} is not true in general.
 \begin{example}
 Let $\Gamma=\mathbb Z$ and $\Omega=\mathbb Z_2$, then $\Gamma$ is a $\Omega$-graded ring with $\Gamma_0=\mathbb Z$ and $\Gamma_1=\{0\}$. Consider  $\Im=\mathbb Z_{pq} \oplus \mathbb Q$ as a $\mathbb Z$-module, where $p,q$ are two prime integers, $\Im$ is a $\Omega$-graded module with $\Im_0=\mathbb Z_{pq} \oplus \{0\}$ and $\Im_1=\{\Bar{0}\}\oplus \mathbb Q$. Then $Ann_\Gamma(\Im)=\{0\}$ is a $gr$-$2^{A}$-primary ideal of $\mathbb Z$. But $\Im$ is not $gr$-$C$-$2^{A}$-secondary $\mathbb Z$-module, since
 $pqM \subseteq \{\Bar{0}\}\oplus \mathbb Q,$  but $pM=p\mathbb Z_{pq} \oplus \mathbb Q \nsubseteq \{\Bar{0}\}\oplus \mathbb Q$ and $qM=q\mathbb Z_{pq} \oplus \mathbb Q \nsubseteq \{\Bar{0}\}\oplus \mathbb Q$ and $pq \notin Gr(Ann_\Gamma(\Im))$.

 \end{example}
 A graded domain $\Gamma$ is called a gr-Dedekind ring if every
graded ideal of $\Gamma$ factors into a product of graded prime ideals, see \cite{VF}.

 A graded $\Gamma$-module $\Im$ is called a gr-comultiplication module if for every graded submodule $S$ of $\Im$, there exists a graded ideal $P$ of $\Gamma$ such that
$S=(0:_\Im P)$, equivalently, for each graded submodule $S$ of $\Im$, we have $S=(0:_\Im Ann_\Gamma(S))$, see \cite{AF}.

 The $gr$-$C$-$2^{A}$-secondary submodules of a gr-comultiplication module over a gr-Dedekind domain are described in the following theorem.

 \begin{theorem}
 Let $\Gamma$ be a gr-Dedekind domain, and $\Im$ be a gr-comultiplication $\Gamma$-module, if $S$ is $gr$-$C$-$2^{A}$-secondary submodule of $\Im$, then $S=(0:_\Im Ann_\Gamma ^n(L))$ or $S=(0:_\Im Ann_\Gamma ^n(L_1) Ann_\Gamma ^m(L_2))$, where $L, L_1, L_2$ are graded minimal submodules of $\Im$ and $n, m$ are positive integers.
 \end{theorem}

 \begin{proof}
By Theorem \ref{thm 2.5}, since  $S$ is $gr$-$C$-$2^{A}$-secondary submodule of $\Im$, then $Ann_\Gamma(S)$ is a $gr$-$2^{A}$-primary ideal of $\Gamma$. Using[\cite{DA}, Theorem 4.1] and [\cite{VF}, Lemma 1.1], we have either $Ann_\Gamma(S)=I^n$ or $Ann_\Gamma(S)=I_{1}^n I_{2}^m$, where $I, I_1, I_2$ are graded maximal ideals of $\Gamma$. First assume  $Ann_\Gamma(S)=I^n$. If $(0 :_\Im I)=0$, then $(0 :_\Im I^n)=0$, and so we conclude that $S=0$, a contradiction. Now by [\cite{AF}, Theorem 3.9 ], since  $I$ is graded maximal ideal of $\Gamma$, we have $(0 :_\Im I)$ is graded minimal submodule of $\Im$. This implies that $S=(0 :_\Im Ann_\Gamma^n(L))$, where $L=(0 :_\Im I)$. Now assume that $Ann_\Gamma(S)=I_{1}^n I_{2}^m$. If $(0 :_\Im I_1)=0$ and $(0 :_\Im I_2)=0$, then $S=0$, a contradiction. Thus either $(0 :_\Im I_1) \neq 0$ or $(0 :_\Im I_2) \neq 0$. Hence one can see that either $S=(0 :_\Im Ann_\Gamma^n(L_1)Ann_\Gamma^m(L_2))$ or $S=(0 :_\Im Ann_\Gamma^n(L_1))$ or $S=(0 :_\Im Ann_\Gamma^m(L_2))$, where $L_1=(0 :_\Im I_1)$ and $L_2=(0 :_\Im I_2)$ are graded minimal submodules of $\Im$.
 \end{proof}

 For a graded $\Gamma$-submodule $S$ of $\Im$, the graded second radical of $S$ is defined as the sum of all gr-second $\Gamma$-submodules of $\Im$ contained
in $S$, and its denoted by $GSec(S)$. If $S$ does not contain any gr-second $\Gamma$-submodule, then $GSec(S) =\{0\}$.
  The graded second spectrum of $\Im$ is the collection of all gr-second $\Gamma$-submodules, and it is represented by the symbol $GSpec^s(\Im)$.
  The set of all gr-prime
$\Gamma$-submodules of $\Im$ is called the graded spectrum of $\Im$, and is denoted by $GSpec(\Im)$. The map $\psi : GSpec^s(\Im) \rightarrow GSpec(\Gamma / Ann_\Gamma(\Im))$ defined by
$\psi(S) = Ann_\Gamma(S)/Ann_\Gamma(\Im)$ is called the natural map of $GSpec^s(\Im)$, see \cite{RR}.
Graded submodule $S$ of $\Im$ is called a \textit{a\emph{\ }}graded strongly
2-absorbing second (Abbreviated, $gr$-$S$-$2^{A}$-second) submodule of $\Im$ if whenever $a,$ $b\in h(\Gamma),$ $%
S_{1},S_{2}$ are completely graded irreducible submodules of $\Im$, and $%
abS\subseteq S_{1}\cap S_{2},$ then $aS\subseteq S_{1}\cap S_{2}$ or $%
bS\subseteq S_{1}\cap S_{2}$ or $ab\in Ann_{\Gamma}(S)$, see \cite{KM}.

\begin{theorem}
Let $\Im$ be a $gr$-comultiplication $\Gamma$-module, and the natural map $\psi$ of $GSpec^s(S)$ is surjective, if $S$ is a $gr$-$C$-$2^{A}$-secondary submodule of $\Im$, then $GSec(S)$ is a $gr$-$S$-$2^{A}$-second submodule of $\Im$.
\end{theorem}

\begin{proof}
Let $S$ be a $gr$-$C$-$2^{A}$-secondary submodule of $\Im$. By Corollary \ref{2.6},  $Gr(Ann_\Gamma(S))$ is $gr$-$2^{A}$ ideal of $\Gamma$.
By[\cite{RR}, Lemma 4.7], $Gr(Ann_\Gamma(S))=Ann_\Gamma (GSec(S))$. Therefore, $Ann_\Gamma (GSec(S))$ is $gr$-$2^{A}$ ideal of $\Gamma$. Using[\cite{RR}, Proposition 3.7],  $GSec(S)$ is $gr$-$S$-$2^{A}$-second $\Gamma$-submodule of $\Im$.
\end{proof}

 Let $\Gamma$ be a $\Omega$-graded ring, a graded $\Gamma$-module $\Im$ is a gr-sum-irreducible if $\Im \neq 0$ and the sum of any two proper graded submodule of $\Im$ is always a proper graded submodule, see  \cite{AR}.

\begin{theorem}
Let $S$ be a $gr$-$C$-$2^{A}$-secondary submodule of $\Im$. Then $r \, S=r^2 \, S,  \forall r \in h(\Gamma) \backslash Gr(Ann_\Gamma(S))$.
The converse hold, if $S$ is a gr-sum-irreducible submodule of $\Im$.
\end{theorem}

\begin{proof}
Let $r \in h(\Gamma) \backslash Gr(Ann_\Gamma(S))$. Then  $r^2 \in h(\Gamma) \backslash Gr(Ann_\Gamma(S))$. Thus by Theorem \ref{thm2.4}, we have $r \, S=r^2 \, S$. Conversely, let  $S$ be a gr-sum-irreducible submodule of $\Im$ and $rs \, S \subseteq L$, for some $r, s \in h(\Gamma)$ and $L\leq _{\Omega}\Im$. Suppose that $rs \notin Gr(Ann_\Gamma(S))$. We show that $r \, S \subseteq L$ or $s \, S \subseteq L$. Since $rs \notin Gr(Ann_\Gamma(S))$, we have $r,s \notin Gr(Ann_\Gamma(S))$. Thus $r \, S=r^2 \, S$ by assumption. Let $x \in S$, then $rx \in r \, S = r^2 \, S$. So $\exists y \in S$ such that $r x=r^2 y$. This implies that $x-r y \in (0:_S r) \subseteq(L :_S r)$. Thus $x=x-r y+ r y \in (L :_S r)+(L :_S s)$. Hence $S \subseteq (L :_S r)+(L :_S s) $. Clearly, $ (L :_S r)+(L :_S s) \subseteq S $, as  $S$ is gr-sum-irreducible submodule of $\Im$, $ (L :_S r)=S$ or $ (L :_S s)= S$, i.e $r \, S \subseteq L$ or $s \, S \subseteq L$, as needed.
\end{proof}

A graded $\Gamma$-module $\Im$ is called gr-multiplication, if for every graded submodule $S$ of $\Im$, there exists a graded  ideal $K$ of $\Gamma$ such that $S=K\Im$, see \cite{OTO}.

\begin{theorem}
Let $S\leq _{\Omega}\Im$. Then we have the following.

\vspace{0.1cm}
a) If $S$ is a $gr$-$C$-$2^{A}$-secondary submodule of $\Im$, then $IC$ is a $gr$-$C$-$2^{A}$-secondary submodule of $\Im$, for all graded ideal $I$ of $\Gamma$, with $I \nsubseteq Ann_\Gamma(S)$.

\vspace{0.1cm}
b) If $\Im$ is a $gr$-multiplication $gr$-$C$-$2^{A}$-secondary module, then every non-zero graded submodule of $\Im$ is a $gr$-$C$-$2^{A}$-secondary submodule of $\Im$.
\end{theorem}

\begin{proof}
a) Let $I$ be a graded ideal of $\Gamma$, with $I \nsubseteq Ann_\Gamma(S)$. Then $IC$ is a non-zero graded submodule of $\Im$. Let $r, s \in h(\Gamma)$, $L$ is graded submodule of $\Im$, and $rs \, IC \subseteq L$, then $rs \, S \subseteq (L :_\Im I)$, thus $r \, IC \subseteq L$ or $s \, IC \subseteq L$ or $rs \in Gr(Ann_\Gamma(S)) \subseteq Gr(Ann_\Gamma(IC)) $, as desired.

\vspace{0.1cm}
b) This follows from part a.
\end{proof}

\begin{theorem}
Let $\Gamma$ be $\Omega$-graded ring and $\Im, \Im'$ be two graded $\Gamma$-module. Let $\psi : \Im \rightarrow \Im'$ be a graded monomorphism.

\vspace{0.1cm}
a) If $S$ is a $gr$-$C$-$2^{A}$-secondary submodule of $\Im$, then $\psi(S)$ is  a $gr$-$C$-$2^{A}$-secondary submodule of $\Im'$.

b) If $S'$ is a $gr$-$C$-$2^{A}$-secondary submodule of $\psi (\Im)$, then $\psi^{-1} (S')$ is  a $gr$-$C$-$2^{A}$-secondary  submodule of $\Im$.
\end{theorem}

\begin{proof}
a) As $S \neq 0$, and $\psi$ is a graded  monomorphism, we have $\psi(S) \neq 0$, let $r, s \in h(\Gamma)$, $L'\leq _{\Omega} \Im'$, and $rs \, \psi(S) \subseteq L'$. Then $rs \, S \subseteq \psi^{-1}(L')$. Since $S$ is $gr$-$C$-$2^{A}$-secondary submodule of $\Im$, $r \, S \subseteq \psi^{-1}(L')$ or $s \, S \subseteq \psi^{-1}(L')$ or $rs \in Gr(Ann_\Gamma(S))$. Therefore,
$$ r \, \psi(S) \subseteq \psi(\psi^{-1}(L'))= \psi(\Im) \cap L' \subseteq L'$$ or $$  s \, \psi(S) \subseteq \psi(\psi^{-1}(L'))= \psi(\Im) \cap L' \subseteq L'$$ \\ or $rs \in Gr(Ann_\Gamma(\psi(S)))$, as desired.

\vspace{0.2cm}
b) If $\psi^{-1}(S')=0$, then $\psi(\Im) \cap S' = \psi \, \psi^{-1}(S')= \psi(0)=0$. So $S'=0$, a contradiction. Therefore $\psi^{-1}(S') \neq 0$. Let $r, s \in h(\Gamma)$, $L \leq _{\Omega}\Im$, and $rs \, \psi^{-1}(S') \subseteq L$. Then $$ rs \, S'= rs(\psi(\Im) \cap S')= rs \, \psi \, \psi^{-1}(S') \subseteq \psi(L). $$
As  $S'$ is $gr$-$C$-$2^{A}$-secondary submodule of $\psi (\Im)$, $r \, S' \subseteq \psi(L)$ or $s \, S' \subseteq \psi(L)$ or $rs \in Gr(Ann_\Gamma(S'))$. Thus $r \, \psi^{-1}(S') \subseteq \psi^{-1}\psi(L) = L$ or $s \, \psi^{-1}(S') \subseteq \psi^{-1}\psi(L) = L$ or $rs \in Gr(Ann_\Gamma(\psi^{-1}(S')))$, as needed.
\end{proof}

%
%
%
%
%
%
%
%
%
%


\bigskip\bigskip\bigskip\bigskip

\end{document}